%% file: dropstab.tex
\documentclass[reqno]{amsart}
\usepackage{amsmath,amsthm,amsfonts,color}
\usepackage[latin1]{inputenc}
\usepackage[makeroom]{cancel}

\input{commands.tex}

\newtheorem{thm}{Theorem}[section]
\newtheorem{lem}[thm]{Lemma}

\newtheorem{prop}[thm]{Proposition}
\newtheorem{rems}[thm]{Remarks}
\newtheorem{rem}[thm]{Remark}

\DeclareMathAlphabet{\mathpzc}{OT1}{pzc}{m}{it}

\numberwithin{equation}{section}

\begin{document}
\bibliographystyle{plain}

\title[Stability of Round Droplets]{Equilibria and Their Stability for a Viscous Droplet Model}

\author{Patrick Guidotti}
\address{University of California, Irvine\\
Department of Mathematics\\
340 Rowland Hall\\
Irvine, CA 92697-3875\\ USA }
\email{gpatrick@math.uci.edu}

\begin{abstract}
A classical model of fluid dynamics is considered which describes the shape evolution
of a viscous liquid  droplet on a homogeneous substrate. All equilibria are
characterized and their stability is analyzed by a geometric reduction argument.
\end{abstract}

\subjclass[1991]{}

\maketitle

\section{Introduction}
Consider the shape of a droplet of a viscous liquid on a homogeneous substrate and
denote by $\Omega(t)$ the region wetted by the liquid. The droplet can then be
described by a height field $u:\Omega(t)\to \mathbb{R}$ at any given time during the
evolution (at least in the regime of interest). Asymptotic and averaging techniques
yield, on appropriate assumptions (very small and very viscous droplet, see \cite{G78}),
a simplification of Navier-Stokes equations which is considered here. The system in
question reads
\begin{equation}\label{deq}\begin{cases}
 -\Delta u=\lambda&\text{ in }\Omega(t)\text{, for }t>0\, ,\\
 u=0&\text{ on }\partial\Omega(t)\text{, for }t>0\, ,\\
 \int _{\Omega(t)} u\, dx=V_0>0&\text{ for }t>0\, ,\\
 V=F(|\nabla u|)&\text{ on }\partial\Omega(t)\text{, for }t>0\, ,\\
 \Omega(0)=\Omega _0\, .
\end{cases}
\end{equation}
The nonlinearity $F$ drives the evolution, i.e. the contact angle dynamics via the
fourth equation in \eqref{deq} for the front velocity $V$ in outward normal direction
$\nu(t)$ to the surface $\Gamma(t)=\partial\Omega(t)$. It is easily seen that balls
evolve invariantly (maintaining their shape) for these equations and that a radius
$r_e$ is singled out by the the volume conservation constraint combined with any
canonical choice of $F$ such as
$$
 F(s)=s^2-1\text{ or }F(s)=s^3-1\, , \: s>0\, ,
$$
and that such equilibrium ball is stable for the corresponding ordinary
differential equation it satisfies (cf. \cite{G78}). For the purposes of this paper it
will only be assumed that
$$
 F'>0\text{ and that }F(1)=0\, ,
$$
thus effectively prescribing the contact angle below which the droplet would tend to
locally retract and above which it would locally expand. These are the  minimal
qualitative assumptions of any physically relevant model.

It will be shown here that circles are the only equilibria of \eqref{deq} and
that they are stable with respect to any smooth perturbation. In particular, a
perturbed circle will converge exponentially fast back to a circle albeit centered at
a possibly different point. Its radius is, however, uniquely determined by volume
conservation and therefore remains unchanged as compared to that of the circle being
perturbed. The proof of stability hinges on the explicit computation of the
linearization of \eqref{deq} in a circle and on the use of a special nonlinear
coordinate system in the ``space of curves'' suggested by the translation
invariance of the system.

Previous results about this basic model of fluid dynamics include local and global
existence of appropriate weak solutions \cite{KG09,KG11} which would cover instances
where singularity formation can occur (see the numerical experiments of \cite{Gl05})
and a local well-posedness result \cite{EG13} in the category of classical solutions. 

The remainder of the paper is organized as follows. In the next section the results
of \cite{EG13} are briefly summarized in order to introduce the appropriate functional
setup and since they form the starting point for the subsequent analysis. Then
equilibria are characterized and their stability is investigated by a fully explicit
calculation of the linearization of the nonlinear, nonlocal curve evolution described
by the last two equations of \eqref{deq} and obtained by thinking of the other unknown
$u$ as the function of $\Gamma(t)$ determined by solving the first three equations,
combined with the introduction of convenient nonlinear coordinates in the ``space of
curves'' $\Gamma$ with respect of which the linearization in the equilibrium circle
coincides with the computed one. In these nonlinear coordinates the evolution admits
a simplified description of the dynamics obtained by a direct and revealing
exploitation of the translation invariance of the system. The proposed approach has
to be contrasted with the more complicated and widespread approach typically taken
to deal with similar problems sharing some geometric invariance (translation
invariance in the specific case) with \eqref{deq}. Indeed 
center manifold analyses \cite{ES98b} or a generalized principle of
stability \cite{PSZ09} were previously used. Both, albeit to a different degree, are more
involved than the direct and transparent approach suggested in this paper. 
\section{Setup}
In order to reduce system \eqref{deq} to an evolution equation for a simple unknown,
an appropriate parametrization of $\Gamma(t)$ is necessary. To that end, fix a
$\operatorname{C}^\infty$-hypersurface $\Gamma$ close to $\Gamma _0=\partial\Gamma
_0$ in such a way that the latter can be described as a graph in normal direction
over $\Gamma$, i.e.
$$
 \Gamma _0=\big\{ x+\rho _0(x)\nu (x)\,\big |\, x\in\Gamma\big\}\, ,
$$
for some function $\rho _0:\Gamma\to \mathbb{R}$. For technical reasons (better
invariance properties with respect to interpolation) so-called little H\"older spaces
prove a convenient choice of phase space. It is recalled that, for $\alpha\in(0,1)$
and an open subset of $\mathbb{R} ^n$, the space of bounded uniformly H\"older
continuous functions of exponent $\alpha$
$$
\operatorname{BUC}^\alpha(O)=\{ u:O\to \mathbb{R}\, |\, \| u\| _\infty<\infty\, ,\:
[u]_\alpha:=\sup_{x\neq y}\frac{|u(x)-u(y)}{|x-y|^\alpha}<\infty\}
$$
is a Banach space with respect to the norm
$$
 \|\cdot\| _{\alpha,\infty}=\|\cdot\| _\infty +[\cdot]_\alpha\, .
$$
For $k\in \mathbb{N}$ one also defines
$$
 \operatorname{BUC}^{k+\alpha}(O)=\{ u\in \operatorname{BUC}^k(O)\, |\, \partial
 ^\beta u\in \operatorname{BUC}^\alpha(O)\:\forall\:|\beta|=k\}\, ,
$$
which is a Banach space with respect to the norm $\|\cdot\| _{k,\alpha}=\|\cdot\|
_{k,\infty}+\sup _{|\beta|=k}[\partial ^\beta\cdot]_\alpha$. Here it is used that
$$
 \operatorname{BUC}^k(O)=\{ u:O\to \mathbb{R} \, |\, u\text{ is }k\text{-times
     continuously differentiable}\}
$$
and $\|\cdot\| _{k,\infty}=\sup _{|\beta|\leq k}\|\partial ^\beta u\| _\infty$. The
the little H\"older spaces are given by
$$
 h^{k+\alpha}(O)=\text{closure}_{\|\cdot\| _{k,\alpha}}\bigl(r_O \mathcal{S}(\mathbb{R} ^n)\bigr)\, ,
$$
i.e. as the completion of the restriction of smooth rapidly decreasing functions to
the set $O$ with respect to the $\operatorname{BUC}^{k+\alpha}$ topology. These
spaces can all be transplanted on any smooth compact manifold by the use of standard
localization techniques and a smooth partition of unity. This is how the notation
$h^{k+\alpha}(\Gamma)$ should be interpreted. Any given function $\rho\in
h^{2+\alpha}(\Gamma)$ yields a diffeomorphism $\theta _\rho$ between $\Gamma$ and 
$\Gamma_\rho=\big\{ x+\rho(x)\nu (x)\,\big |\, x\in\Gamma\big\}$, which can be
extended to a diffeomorphism of $\mathbb{R}^n$ still denoted by $\theta _\rho$ such
that
$$
 \theta _\rho:\mathbb{R}^n\to \mathbb{R}^n\, ,\: y\mapsto\begin{cases}
 X(y)+\bigl[\Lambda(y)+\varphi\bigl(\Lambda(y)\bigr)\rho\bigl(X(y)\bigr)\bigr]\nu\bigl(X(y)\bigr)\,
 , &y\in\Omega_\Lambda\, ,\\ y\, ,& y\notin\Omega_\Lambda\, ,\end{cases}
$$
where $\varphi$ is a smooth cut-off function, $\Omega_\Lambda$ is a tubular
neighborhood of $\Gamma$ and $\bigl(X(y),\Lambda(y)\bigr)$ are ``tubular
coordinates'' of $y$, i.e. they satisfy
$$
 y=X(y)+\Lambda(y)\nu\bigl( x(y)\bigr)\, .
$$
Then
$$
 \theta _\rho(\Omega)=\Omega_\rho\, ,\: \theta_\rho(\Gamma)=\Gamma_\rho\, ,\:
 \theta_\rho\big | _{\Omega_\Lambda^c}=\operatorname{id}\, .
$$
Clearly the tubular neighborhood is taken as small as the geometry of $\Gamma$
requires in order to obtain a well-defined coordinate system
$\bigl(X(y),\Lambda(y)\bigr)$ and $\rho$ small enough as to ensure that
$\Gamma_\rho\subset\Omega_\Lambda$.  More explicit and quantitative assumptions can
be found in \cite{EG13} but are not needed in the remainder of this paper. It should,
however, be observed that the smallness assumption on $\rho\in h^{2+\alpha}(\Gamma)$
is immaterial since $\Gamma\in \operatorname{C}^\infty$ can be chosen arbitrarily
close to $\Gamma_0$ so that $\rho _0$ will be small. For the purpose of the analysis
to follow this represents no restriction. With the diffeomorphism $\theta _\rho$ in
hand, the first three equations of \eqref{deq} can be pulled back to a fixed domain
$\Omega$ to give 
\begin{equation}\label{deqfix}\begin{cases}
 \mathbb{A}(\rho)=-\theta ^*_\rho\Delta\theta ^\rho_* u=\lambda&\text{ in }\Omega\text{, for }t>0\, ,\\
 u=0&\text{ on }\Gamma\text{, for }t>0\, ,\\
 \int _{\Omega} u|D\theta _\rho|\, dx=V_0>0&\text{ for }t>0\, ,\\
\end{cases}
\end{equation}
by means of the pull-back
$$
 \theta ^*_\rho:h^{2+\alpha}(\Gamma_\rho)\to h^{2+\alpha}(\Gamma)\, ,\: u\mapsto
 v=u\circ\theta_\rho\, ,
$$
and the associated push-forward $\theta ^\rho_*$ given by its inverse. This can be
done on the assumption that $\Omega(t)=\Omega_{\rho(t,\cdot)}$ for a given
$$
 \rho:[0,T]\to h^{2+\alpha}(\Gamma)\, ,
$$
yielding $\Gamma(t)=\Gamma _{\rho(t,\cdot)}$ and satisfying
$$
 \rho(t,\cdot)\in \mathbb{B}_{h^{1+\alpha}(\Gamma)}( 0,\delta)\cap
 h^{2+\alpha}(\Gamma)=:\mathcal{V}\text{ for }t\in[0,T]\, ,
$$
for a small $\delta>0$ as described in \cite{EG13}. It is useful to denote the solution
of \eqref{deqfix} obtained by fixing $\rho\in \mathcal{V}$ and setting $\lambda=1$ by 
$$
 \theta ^*_\rho u_\rho=v_\rho=\mathbb{S}(\rho)\mathbf{1}\, ,
$$
where $\mathbb{S}(\rho)=\lambda(\rho)\overline{\mathbb{S}}(\rho)$ for
$\theta ^*_\rho\bar u_\rho=\bar v_\rho=\overline{\mathbb{S}}(\rho)\mathbf{1}$ which
solves 
$$
 \begin{cases}\mathbb{A}(\rho)v=1&\text{ in }\Omega\text{ for }t>0\, ,\\
 v=0&\text{ on }\Gamma\text{ for }t>0\, ,\end{cases}\text{ or, equivalently, }
 \begin{cases}-\Delta u=1&\text{ in }\Omega_\rho\text{ for }t>0\, ,\\
 u=0&\text{ on }\Gamma_\rho\text{ for }t>0\, ,\end{cases}
$$
in the original coordinates, and 
$$
\lambda(\rho)=\frac{V_0}{\int _\Omega\bar v_\rho|D\theta
  _\rho|\, dy}=\frac{V_0}{\int _{\Omega_\rho}\bar u\, dx}\, ,
$$
rescales the solution to satisfy the volume constraint. As shown in \cite{EG13}, it is 
observed in passing that operators or other quantities depending on $\rho$, do indeed
depend analytically on it. This fact will be needed later. In order to reformulate
the kinematic equation in \eqref{deq}, it is convenient to define
$$
 N_\rho:\Omega_\Lambda\to \mathbb{R}\, ,\: y\to\Lambda(y)-\rho\bigl(X(y)\bigr)\, ,
$$
so that
$$
 \Gamma_\rho=N_\rho^{-1}(0)\text{ and }\nu_\rho(y)=\frac{\nabla N_\rho(y)}{|\nabla
   N_\rho(y)|}\, ,\: y\in \Gamma_\rho\, .
$$
Then it can be written
$$
 V(y,t)=\frac{\rho_t}{|\nabla N_\rho(y)|}=F(|Du_\rho|)\, .
$$
Observing that $|Du|=-\partial _\nu u$ on $\Gamma_\rho$ in this particular case, the
following single scalar nonlinear, nonlocal evolution equation for the ``shape
function'' $\rho$ results
\begin{equation}
\label{nnee}\begin{cases}
\rho_t=|\nabla N_\rho|F(-\partial _{\nu_\rho}u_\rho)=|\nabla
N_\rho|F\bigl(-\lambda(\rho)\partial _{\nu_\rho}\bar u_\rho\bigr)=:G(\rho)\, & \text{ on
}\Gamma\text{ for }t>0\, ,\\
\rho(0)=\rho_0\, &\text{ on }\Gamma\, .\end{cases}
\end{equation}
It follows from \cite{EG13} that \eqref{nnee} is equivalent to the original problem in
the context of classical solutions, that $G$ depends analytically on $\rho$ and that
\begin{thm}
Given $\rho _0\in \mathcal{V}$, there exists $T>0$ and a unique solution 
$$
\rho\in \operatorname{C}^1\bigl([0,T],h^{1+\alpha}(\Gamma)\bigr)\cap
\operatorname{C}\bigl([0,T],h^{2+\alpha}(\Gamma)\bigr)
$$
of \eqref{nnee} and, thus, a solution $\bigl(
\mathbb{S}(\rho)\mathbf{1},\Omega_\rho\bigr)$ of \eqref{deq}. 
\end{thm}
The proof relies on localization, perturbation, and optimal regularity results for
parabolic equations which make it possible to reduce local well-posedness to
properties of the linearization $DG(\rho_0)$ in the initial datum. The abstract
approach of \cite{EG13} is necessary to deal with the most general case but does not
provide any explicit representation for the linearization. It is qualitative in
nature and only yields that $DG(\rho_0)$ is the infinitesimal generator of an
analytic semigroup on $h^{1+\alpha}(\Gamma)$ with domain $h^{2+\alpha}(\Gamma)$.
Here and in order to obtain the stability results outlined earlier a much more
detailed understanding of $DG(\rho_0)$ is necessary in the special case when
$\Gamma=\Gamma _0=\mathbb{S}_{r_e}\, .$ 
\section{Equilibria}
While the results remain valid for any space dimension, the analysis would have to 
be adapted to the specific dimension considered. The technique would essentially 
coincide in all dimensions but the specific spherical functions involved would have 
to be chosen depending on the dimension. In order to avoid rendering the presentation
unnecessarily cumbersome, the choice is made to consider the case $n=2$ in this 
paper.\\
An equilibrium solution of \eqref{deq} is obtained if $(u_e,\Omega_e)$ can be found
such that
$$
 V=F(|Du_e|)=0\, .
$$
On the assumptions made earlier this is the case only if
$$
 \partial _\nu u_e=-|Du_e|\equiv 1\text{ on }\Gamma_0\, ,
$$
so that $u_e$ satisfies
$$
\begin{cases}
 -\Delta u_e=\lambda\, ,&\text{ in }\Omega\, ,\\
 u_e=0\, ,&\text{ on }\Gamma\, ,\\
 \partial _\nu u_e=-1\, ,&\text{ on }\Gamma_0\, .
\end{cases}
$$
A classical rigidity result by Serrin \cite{Se71} then implies that a classical solution
of the above overdetermined system can only exist if $\Omega_e$ is a circle of
some radius $r_e$ determined by the additional requirement that
$$
 \int _{\Omega_e}u_e\, dx=V_0\, .
$$
\begin{thm}
If $(u_e,\Omega_e)$ is a steady-state of \eqref{deq}, then $\Omega_e$ must be a 
circle $\mathbb{S}_{r_e}$ of radius
$$
 r_e=\sqrt[3]{\frac{4V_0}{\pi}}\text{ and }u_e=\frac{1}{2}(r_e-\frac{r^2}{r_e})\, ,
$$
where $r$ is the distance from the center of the circle. The parameter
$\lambda$ satisfies $\lambda=2/r_e$.
\end{thm}
\begin{proof}
Serrin's classical result implies that $\Omega_e$ is a sphere. The rest follows from
a direct computation.
\end{proof}
\begin{rems}\label{rems}
{\bf  (a)} The author of \cite{G78} include a partial stability result. They fix a
center of the circle and derive and ode describing the evolution of a circle
of initial radius $r_0$ with the same center. They show that the circle of
radius $r_e$ is locally asymptotically stable.\\
{\bf  (b)} In order to obtain a more general stability result the center and, more,
in general, the geometry needs to be perturbed as well. The ``freedom in the choice
of center'' is responsible for the existence of a (translational) eigenvalue
$0\in\sigma\bigl(DG(\rho_e)$ where $\rho _e\equiv 0$ when the reference manifold is
the stationary solution $\mathbb{S}_{r_e}$ itself.\\
{\bf  (c)} The functions in the kernel of $DG(0)$ can be computed by parametrizing
the shifted circle
$$
 \Omega _{\epsilon v}=\mathbb{S}_{r_e}+\epsilon v\, ,
$$
over the stationary reference circle which can be assumed to be centered in the
origin without loss of generality. To do so, a function
$h_\epsilon:\Omega\to\mathbb{R}$ needs to be determined such that
$$
 (r_e+h_\epsilon\bigr)\begin{bmatrix}\cos(\theta)\\\sin(\theta)\end{bmatrix}-\epsilon
 \begin{bmatrix}v_1\\v_2\end{bmatrix}=r_e \begin{bmatrix}\cos(\phi)\\\sin(\phi)\end{bmatrix}
 \, . 
$$
This yields
$$
 h_\epsilon(\theta)=-r_e+\sqrt{r_e^2+\epsilon^2|v|^2+2\epsilon r_ev\cdot \nu_e}\, ,
$$
where
$$
 \nu _e=\begin{bmatrix}\cos(\theta)\\\sin(\theta)\end{bmatrix}\, ,\: \tau
 _e=\begin{bmatrix}-\sin(\theta)\\\cos(\theta)\end{bmatrix} 
$$
can be thought of as the unit outward normal and unit tangent to the circle
$\mathbb{S}_{r_e}$, respectively. The angle $\theta$ satisfies
$$
 \tan(\theta)=\frac{r_e\sin(\phi)+\epsilon v_2}{r_e\cos(\phi)+\epsilon v_1}\, ,\:
 \arctan(\theta)=\frac{r_e\cos(\phi)+\epsilon v_1}{r_e\sin(\phi)+\epsilon v_2}\, ,
$$
and, for $\epsilon<<1$, is uniformly close to $\phi$, i.e.
$$
 \theta _\epsilon(v_1,v_2)(\phi)\to\phi\text{ as }\epsilon\to 0\, ,
$$
uniformly in $v\in \mathbb{S}_1$. It follows that
$$
\frac{d}{d\epsilon}\big |_{\epsilon=0}h_\epsilon(\theta_\epsilon)=\frac{1}{2r_e}2r_e\nu _e\cdot
v=v_1\cos(\theta)+v_2\sin(\theta)
$$
is an element of the kernel of $DG(0)$ for any $\mathbb{R}^2$.\\
{\bf  (d)} It will be shown that the kernel only consists of the above
``translational'' eigenvectors. Translation also yields a manifold of
equilibria $\mathcal{E}$ (the set of all circle with fixed radius $r_e$) which can
locally be parametrized as in (c). It in fact corresponds to a global center manifold
for the evolution. The analysis which follows does, however, not make use of 
any abstract results about center manifolds but rather exploits directly the
translation invariance. 
\end{rems}
\section{Linearization}
Consider now the equilibrium $\mathbb{S}_{r_e}$ centered at the origin and use it as
the reference manifold $\Gamma$ so that the equilibrium  solution $\rho _e$ vanishes
identically. Then
$$
 X(y)=r_e \frac{y}{|y|}\, ,\: \Lambda(y)=|y|-r_e\, ,\: \nu _e=\frac{y}{|y|}\, ,
$$
and $y=X(y)+\Lambda(y)\frac{y}{|y|}$. In this case the function $N_\rho$ is simply
given by
$$
 N_\rho(y)=|y|-r_e-\rho(y/|y|)\, .
$$
For ease of computation and notation the Euclidean coordinate $y$ or the polar
$(r,\phi)$ will be used interchangeably. In particular, functions on
$\mathbb{S}_{r_e}$ will be identified with functions of the angle variable
$\phi$. With this convention one has
$$
 \nabla N_\rho=1 \frac{\partial}{\partial
   r}-\rho'(\phi)\frac{\partial}{\partial\phi}\text{ and } |\nabla
 N_\rho|^2=(r_e+\rho)^2+\rho'(\phi)^2\, .
$$
It can easily be seen that
$$
 \frac{d}{d\epsilon}\big |_{\epsilon =0} \frac{|\nabla N_{\epsilon h}|}{r_e+\epsilon
   h}=0\, ,\: h\in h^{2+\alpha}(\Gamma)\, .
$$
When computing the linearization it is therefore possible to replace any occurrence
of $\frac{|\nabla N_{\epsilon h}|}{r_e+\rho}$ by its value in 
$\epsilon=0$ and thus consider only $F\bigl(-\lambda(\epsilon h)\partial
_{\nu_{\epsilon h}}\bar u_{\epsilon h}\bigr)$ which leads to
\begin{equation}\label{lin}
 \frac{d}{d\epsilon}\big |_{\epsilon =0}G(\epsilon h)=-F'(1)\frac{d}{d\epsilon}\big
 |_{\epsilon =0} \bigl[\lambda(\epsilon h)\partial _{\nu_{\epsilon h}}\bar u_{\epsilon h}\bigr]
\end{equation}
It turns out, contrary to the approach taken in \cite{EG13}, that it is more convenient
not to perform the transformation to a fixed domain when computing the linearization
in a circle.
\begin{thm}
Let $h\in h^{2+\alpha}(\Gamma)$. Then
$$
\frac{d}{d\epsilon}\big |_{\epsilon =0}\partial _{\nu _{\epsilon h}}\bar u_{\epsilon
  h}=\frac{4V_0}{\pi r_e^4}\bigl[ r_eDTN_{\mathbb{S}_{r_e}}(h)-h\bigr]\, ,
$$
where $DTN_{\partial\Omega}$ is the so-called Dirichlet-to-Neumann operator,
i.e. the operator mapping a Dirichlet datum $h$ to the outward normal derivative
$\partial _\nu w_h$ of the solution $w_h$ of the boundary value problem
$$
\begin{cases}-\Delta w=0\, ,&\text{ in }\Omega\, ,\\
w=h\, ,&\text{ on }\partial\Omega\, .
\end{cases}
$$ 
\end{thm}
\begin{rem}
The above theorem provides a formula for the first ``domain'' variation of the
solution of
\begin{equation}\label{bvp1}
\begin{cases}
-\Delta u=1\, ,&\text{in }\Omega\, ,\\
u=0\, ,&\text{on }\partial \Omega\, ,
\end{cases}
\end{equation}
in the circle of radius $r_e$, i.e. $\frac{\partial}{\partial\Omega}u\big
|_{\Omega=\mathbb{S}_{r_e}}$.
\end{rem}
\begin{proof}
Consider the solution $\bar u_{\epsilon h}$ of \eqref{bvp1} for
$\Omega=\Omega_{\epsilon h}$ and look for it in the form
$$
 \bar u_{\epsilon h}=w_{\epsilon h}+\frac{1}{2}\bigl(r_e^2-|x|^2\bigr)\, .
$$
Then clearly
$$
 -\Delta \bar u_{\epsilon h}=-\Delta w_{\epsilon h}+1=1\text{ or }-\Delta w_{\epsilon
   h}=0\, ,
$$
and 
$$
 \bar u_{\epsilon h}\big |_{\Gamma _{\epsilon h}}=w_{\epsilon h}\big |_{\Gamma
 _{\epsilon h}}+ \frac{1}{2}\bigl( r_e^2-(r_e+\epsilon h)^2\bigr)\text{ or
 }w_{\epsilon h}\big |_{\Gamma _{\epsilon h}} =\epsilon r_eh+\frac{\epsilon
   ^2}{2}h^2\, .
$$
It follows that
\begin{equation*}
 \partial _{\nu _{\epsilon h}}\bar u _{\epsilon h}=\partial _{\nu _{\epsilon h}}w
 _{\epsilon h}+\partial _{\nu _{\epsilon h}}\frac{1}{2}\bigl( r_e^2-|x|^2\bigr)=r_e
 DTN_{\Omega _{\epsilon h}}\bigl( \epsilon h+\frac{\epsilon
   ^2}{2r_e}h^2\bigr)-(r_e+\epsilon h)\nu _{\epsilon h}\cdot \nu _e\, . 
\end{equation*}
Since $\nu _{\epsilon h}=\bigl(-\epsilon h'\tau _e+(r_e+\epsilon h)\nu _e\bigr)/\sqrt{
(r_e+\epsilon h)^2+\epsilon ^2h'^2}$ one has that
$$
 \partial _{\nu _{\epsilon h}}\bar u _{\epsilon h}=r_e DTN_{\Omega _{\epsilon
 h}}\bigl( \epsilon h+\frac{\epsilon ^2}{2r_e}h^2\bigr)-
 \frac{(r_e+\epsilon h)^2}{\bigl[ (r_e+\epsilon h)^2+\epsilon ^2h'^2\bigr]^{1/2}}\, .
$$
Now, if $DTN_{\Omega _\rho}$ depends continuously on $\rho$, it can be easily
inferred that
$$
 \frac{d}{d\epsilon}\big |_{\epsilon =0}\partial _{\nu _{\epsilon h}}\bar u
 _{\epsilon h}=r_eDTN_{\mathbb{S}_{r_e}}(h)-h\, .
$$
The continuous dependence, however, follows from
$$
 DTN_{\Omega _\rho}(h)=\partial _{\nu_{\rho}}\bar u_\rho=\frac{\nabla N_\rho}{|\nabla
   N_\rho|}\cdot \nabla \bar u_\rho\big |_{\Gamma _\rho}\, ,
$$
for
$$
 \bar u_\rho=-\bigl( \theta ^*_\rho\Delta _{\mathbb{S}_{r_e}}\theta ^\rho_*,\gamma
 _{\Gamma _\rho}\bigr) ^{-1}(1,0)\, ,
$$
for $\rho\in \mathcal{V}$. It can be seen as in \cite{EG13} that
$$
 \bigl[\rho\mapsto DTN_{\Omega _\rho}\bigr]: h^{2+\alpha}(\Gamma _\rho)\to
 \mathcal{L}\bigl( h^{2+\alpha}(\Gamma _\rho),h^{1+\alpha}(\Gamma _\rho)\bigr)
$$
is an analytic function because $N_\rho$ and $\theta _\rho$ depend algebraically on
$\rho$. Notice that $h:\mathbb{S}_{r_e}\to \mathbb{R}$, that is, a function of the
angle variable $\phi$ only, can always be transplanted on $\Gamma_\rho$ to and
identified with the function $\tilde h:\Gamma _\rho\to \mathbb{R}$ via
$$
 \tilde h\Bigl( \bigl[(r_e+\rho(\theta)\bigr]\nu _e(\theta)\Bigr)=h(\theta)\, ,\: \theta\in[0,2\pi)\, .
$$
Thus the operator $DTN_{\Omega _\rho}$ can be viewed as defined on the fixed 
space $h^{2+\alpha}(\mathbb{S}_{r_e})\widehat{=}h^{2+\alpha}_p$, where the latter is
the space of $2\pi$-periodic little H\"older functions.
\end{proof}
In order to complete the evaluation of the linearization, the term
$$
\frac{d}{d\epsilon}\big |_{\epsilon =0}\lambda(\epsilon)=\frac{d}{d\epsilon}\big
|_{\epsilon =0}\frac{V_0}{\int _{\Omega _{\epsilon h}}u_{\epsilon h}\, dx}\, .
$$
needs to be evaluated. Thus consider
$$
 \frac{d}{d\epsilon}\big |_{\epsilon =0}\int _{\Omega _{\epsilon h}}u_e(x)\,
 dx=\int _{\Gamma _{\epsilon h}}u_e(x)\, d\sigma _{\Gamma _{\epsilon h}}(x)=0\, , 
$$
by the boundary condition, so that
$$
 \frac{d}{d\epsilon}\big |_{\epsilon =0}\int _{\Omega _{\epsilon h}}u_{\epsilon h}(x)\,
 dx=\int _{\mathbb{B}(0,r_e)}\frac{d}{d\epsilon}\big |_{\epsilon =0}u_{\epsilon h}(x)\,
 dx\, .
$$
Using a Green's function representation for the solution, i.e.
$$
 u_{\epsilon h}(x)=\int _{\Omega _{\epsilon h}}G_{\epsilon h}(x,\bar x)\, d\bar x\, , 
$$
this amounts to computing
$$
 \frac{d}{d\epsilon}\big |_{\epsilon =0}u_{\epsilon h}=\int _{\mathbb{S}_{r_e}}G_0(x,\bar x)\,
 d\sigma _{\mathbb{S}_{r_e}}(\bar x)+\int _{\mathbb{B}(0,r_e)}\frac{d}{d\epsilon}\big
 |_{\epsilon =0}G_{\epsilon h}(x,\bar x)\, d\bar x=\int
 _{\mathbb{B}(0,r_e)}\frac{d}{d\epsilon}\big |_{\epsilon =0}G_{\epsilon h}(x,\bar x)\, d\bar x\, .
$$
Notice that the boundary integral term vanishes because the Green's function for the
Dirichlet problem is zero on the boundary.
For the last term it is resorted to the so-called Hadamard domain variation formula
(see \cite{S05} for a generalized version and more recent developments) for Green's
functions which, in this particular case, yields 
$$
 G_{\epsilon h}(x,\bar x)-G_e(x,\bar x)=\epsilon\int _0^{2\pi}\partial
 _rG_e(x,r_e,\phi)\partial _rG_e(\bar x,r_e,\phi)h(\phi)\, d\phi+o(\epsilon)\, ,
$$
for any $h\in \operatorname{C}^{2+\alpha}(\mathbb{S}_{r_e})$. In order to continue the
computation it is convenient to have an explicit formula for the Dirichlet Green's
function for the circle of radius $r_e$
$$
 G(r,\theta,\bar r,\bar\theta)=\frac{1}{4\pi}\log\bigl[\frac{r_e^2r^2+r_e^2\bar
   r^2-2r_e^2r\bar r\cos(\theta-\bar\theta)}{r^2\bar r^2+r_e^4-2r_e^2r\bar
   r\cos(\theta-\bar\theta}\bigr] 
$$
from which it follows that
$$
 \partial _rG(r,\theta,r_e,\phi)=
 \frac{r_e^2}{2\pi}\frac{r_e^2-r^2}{r_e^2r^2+r_e^4-2r_e^3r\cos(\theta-\bar\theta)}=
 \frac{1}{2\pi}\frac{r_e^2-r^2}{r^2+r_e^2-2r_er\cos(\theta-\bar\theta)}\, .
$$
It is important to observe that that the function
$$
 (r,\theta)\mapsto \int _0^{2\pi}\partial _rG(r,\theta,r_e,\phi)g(\phi)\, d\phi
$$
is harmonic in $\mathbb{B}(0,r_e)$ with boundary value $g$ on
$\mathbb{S}_{r_e}$. Combining everything it can see that
\begin{multline*}
 \int _{\mathbb{B}(0,r_e)}\int _{\mathbb{S}_{r_e}}\frac{d}{d\epsilon}\big |_{\epsilon
   =0}G_{\epsilon h}(x,\bar x)\, d\bar x\\=\int _0^{2\pi}\int _0^{r_e}\int
 _0^{2\pi}\frac{1}{2\pi}\frac{r_e^2-r^2}{r^2+r_e^2-2r_er\cos(\theta-\phi)}\frac{1}{2\pi}
 \frac{r_e^2-\bar r^2}{\bar r^2+r_e^2-2r_e\bar r\cos(\bar\theta-\phi)}h(\phi)\,
 d\phi\bar rd\bar rd\bar\theta\\=\int _{\mathbb{B}(0,r_e)}\int
 _0^{2\pi}\frac{1}{2\pi}\frac{r_e^2-r^2}{r^2+r_e^2-2r_er\cos(\theta-\phi)}\frac{r_e^2}{2}\,
 d\phi rdrd\theta
\\=\frac{r_e^2}{2}\int _0^{2\pi}h(\phi)\int _0^{2\pi}\int _0^{r_e}
\frac{1}{2\pi}\frac{r_e^2-r^2}{r^2+r_e^2-2r_er\cos(\theta-\phi)}\, rdrd\theta d\phi
=\frac{r_e^4}{4}\int _0^{2\pi}h(\phi)\, d\phi=\frac{\pi r_e^4}{2}\hat h_0\, ,
\end{multline*}
where $\hat h_0$ is the average of the function $h$. Returning to the computation of
the linearization it is seen that
\begin{equation*}
 \frac{d}{d\epsilon}\big |_{\epsilon =0}\frac{V_0}{\int _{\Omega _{\epsilon
       h}}u_{\epsilon h}\, dx}=-\frac{V_0}{\bigl(\int _{\Omega _{\epsilon
       h}}u_{\epsilon h}\, dx\bigr)^2 }\frac{d}{d\epsilon}\big |_{\epsilon =0}\int
 _{\Omega _{\epsilon h}}u_{\epsilon h}\, 
 dx=-\frac{16V_0}{\pi ^2r_e^8}\frac{\pi r_e^4}{2}\hat h_0=-\frac{8V_0}{\pi r_e^4}\hat h_0\, .
\end{equation*}
This concludes the computation of the linearization which is summarized in the next
theorem.
\begin{thm}
For $h\in h^{2+\alpha}(\mathbb{S}_{r_e})$, it holds that
\begin{equation}\label{linrep}
 \frac{d}{d\epsilon}\big |_{\epsilon =0}G(\epsilon h)=-F'(1)\Bigl(\frac{4V_0}{\pi
   r_e^4}\bigl[ r_eDTN_{\mathbb{S}_{r_e}}(h)-h\bigr] +\frac{8V_0}{\pi r_e^4}\hat h_0\Bigr)
\end{equation}
\end{thm}
\begin{proof}
The calculations preceding the formulation of the theorem yield a complete proof by
observing that \eqref{lin} implies that
$$
\frac{d}{d\epsilon}\big |_{\epsilon =0}G(\epsilon h)=-F'(1)\bigl[\frac{d}{d\epsilon}\big
 |_{\epsilon =0}\lambda(\epsilon h)\partial _{\nu_e}\bar
 u_e+\lambda(0)\frac{d}{d\epsilon}\big |_{\epsilon =0}\partial _{\nu_{\epsilon
     h}}\bar u_{\epsilon h}\bigr] 
$$
and also that $\partial _{\nu _e}\bar u_e\equiv -1$.
\end{proof}
Exploiting an alternative representation for the solution $w_h$ of
$$
\begin{cases}
 -\Delta w=0\, ,&\text{in }\mathbb{B}(0,r_e)\, ,\\
 w=h\, ,&\text{on }\mathbb{S}_{r_e}\, ,
\end{cases}
$$
given by
$$
 w_g=\frac{1}{2\pi}\hat g_0+\sqrt{2}\sum_{k=1}^\infty \frac{r^k}{r_e^k}\bigl[\hat
 h^c_k\cos(k\theta)+\hat h^s_k\sin(k\theta)\bigr]\, ,
$$
it is arrived at
$$
 DTN_{\mathbb{S}_{r_e}}(h)=\partial _{\nu _e}w_h=\frac{\sqrt{2}}{r_e}\sum _{k=1}^\infty
 k\bigl[\hat h^c_k\cos(k\theta)+\hat h^s_k\sin(k\theta)\bigr]\, ,
$$
where $\hat h^c_k\, ,\: \hat h^s_k$ are the Fourier coefficients of the function $h$
with respect to the orthonormal basis
$$
 \frac{1}{2\pi}, \sqrt{2}\cos(\theta), \sqrt{2}\sin(\theta), \sqrt{2}\cos(2\theta),
 \sqrt{2}\sin(2\theta), \dots
$$
of $\operatorname{L}^2(\mathbb{S}_{r_e})$. Together with representation
\eqref{linrep} this yields
\begin{thm}
The spectrum of the linearization is given by
$$
 \sigma\bigl(DG(0)\bigr)=-F'(1)\frac{4V_0}{\pi r_e^2}\{0,1,2,3,\dots\}\, .
$$
\end{thm}
The kernel is precisely the two-dimensional space generated by $\cos(\theta)$ and
$\sin(\theta)$ due to the translation invariance of the problem as observed in
Remarks \ref{rems} (c) and (d). The first negative eigenvalue has eigenspace
generated by the functions $\mathbf{1},\cos(2\theta),\sin(2\theta)$, whereas the
remaining negative eigenvalues corresponding to $k=2,3,\dots$ have
eigenspace generated by
$\cos\big((k+1)\theta\bigr),\sin\bigl((k+1)\theta\bigr)$. With 
this knowledge of the linearization in hand, it would be possible to apply either a
general center-manifold reduction approach \cite{S95} or the generalized principle of
stability proved in \cite{PSZ09}. It is arguable that more insight is, however,
gained by the direct approach taken in the next section.
\section{Stability Analysis}
For the purpose of analyzing the stability of equilibria it is more convenient to use 
a slightly different parametrization of the manifold of curves about a fixed
steady-state. Any small enough $\rho\in \mathcal{V}$ can be described using the 
coordinates
$$
 \bigl( v_1(\rho),v_2(\rho),\bar\rho(\rho)\bigr)
$$
where $v=(v_1,v_2)$ is the spatial location coordinate and $\bar\rho$ is the ``shape
coordinate''. In other words
$$
\bigl( r_e+\rho(\theta)\bigr)
\begin{bmatrix}\cos(\theta)\\\sin(\theta)\end{bmatrix}
=\begin{bmatrix}v_1\\v_2\end{bmatrix}
+\bigl(r_e+\bar\rho(\theta)\bigr)\begin{bmatrix} \cos(\theta)\\\sin(\theta)\end{bmatrix}
$$
where $v$ is chosen so that $\bar\rho\in N\bigl(DG(0)\bigr)^\perp$. The main reason
for the use of this coordinate system is the fact that
\begin{lem}
It holds that $F(\rho)=F(\bar\rho)$.
\end{lem}
\begin{proof}
Given $\rho\in \mathcal{V} $ consider the domains $\Omega _\rho$ and $v+\Omega_\rho$
for $v\in \mathbb{R}^2$. The solution $u_v$ of
$$\begin{cases}
-\Delta u_v=\lambda&\text{in }v+\Omega_\rho\, ,\\
u_v=0&\text{on }\Gamma_\rho\, ,\\
\int _{\Omega_\rho}u_v(x)\, dx=V_0\, ,&
\end{cases}
$$
clearly satisfies
$$
 u_v(x)=u_0(x-v)\, ,\: x\in\rho+\Omega
$$
so that
$$
 \partial _{\nu _{v+\Gamma_\rho}}u_v=\partial _{\nu _{\Gamma_0}}u_0(\cdot -v)\, .
$$
If $\theta$ is the angle variable, then
$$
 \partial _{\nu _{v+\Gamma_\rho}}u_v(\theta)=\partial _{\nu _{\Gamma_0}}u_0(\theta)\, ,
$$
and so $F(\rho)=F(v,\bar\rho)=F(\bar\rho)$.
\end{proof}
In these coordinates
$$
 V_\rho=\bigl( \dot
 v+\dot{\bar\rho}(\theta)\begin{bmatrix}\cos(\theta)\\\sin(\theta)\end{bmatrix}\,
 \big |\, \nu _\rho\bigr)
$$
while
$$ 
 \nu_\rho=\frac{1}{\sqrt{(r_e+\bar\rho)^2+\bar\rho'^2}}\bigl[(r_e+\bar\rho)\nu
 _e-\rho'\tau _e\bigr] \, .
$$
It follows that
\begin{equation}\label{nspeed}
 \Bigl[1+\bigl(\frac{\bar\rho'}{r_e+\bar\rho}\bigr)^2\Bigr]^{1/2}V_\rho=\dot
 v_1\bigl[\cos(\theta)-\frac{\bar\rho'}{r_e+\bar\rho}\sin(\theta)\bigr]+\dot
 v_2\bigl[\sin(\theta)+\frac{\bar\rho'}{r_e+\bar\rho}\cos(\theta)\bigr]+\dot{\bar\rho}\, .
\end{equation}
Now, denoting by $\pi ^c_1$, $\pi _1^s$, and $\pi _1^\perp$ the (orthogonal)
projections onto $\mathbb{R}\cos(\theta)$, $\mathbb{R}\sin(\theta)$, and the
orthogonal complement of $\mathbb{R}\cos(\theta)\oplus\mathbb{R}\sin(\theta)$,
respectively, \eqref{nspeed} entails that
$$\begin{cases}
\dot v_1\Bigl\{1-\sqrt{2}\int
_0^{2\pi}\frac{\bar\rho'}{r_e+\bar\rho}\sin(\theta)\cos(\theta)\, d\theta\Big\}+\dot
v_2\Big\{\sqrt{2}\int _0^{2\pi}\frac{\bar\rho'}{r_e+\bar\rho}\cos^2(\theta)\, d\theta\Big\}=\pi
^c_1G(\bar\rho)\, ,\\[0.2cm]
\dot v_1\Big\{\sqrt{2}\int _0^{2\pi}\frac{\bar\rho'}{r_e+\bar\rho}\sin^2(\theta)\,
d\theta\Big\}+\dot v_2\Bigl\{1+\sqrt{2}\int
_0^{2\pi}\frac{\bar\rho'}{r_e+\bar\rho}\cos(\theta)\sin(\theta)\, d\theta\Big\}=\pi
_1^sG(\bar\rho)\, ,\\[0.2cm]
\dot{\bar\rho}=\pi _1^\perp G(\bar\rho)+\dot v_1\pi ^\perp
_1\bigl(\frac{\bar\rho'}{r_e+\bar\rho}\sin(\theta)\bigr)-\dot v_2\pi
_1^\perp\bigl(\frac{\bar\rho'}{r_e+\bar\rho}\cos(\theta)\bigr)\, .
\end{cases}
$$
Defining the matrix $M(\bar\rho)$ by
\begin{multline*}
 M(\bar\rho)=\begin{bmatrix} 1&0\\0&1\end{bmatrix}+\begin{bmatrix} -\sqrt{2}\int
_0^{2\pi}\frac{\bar\rho'}{r_e+\bar\rho}\sin(\theta)\cos(\theta)\, d\theta
&\sqrt{2}\int _0^{2\pi}\frac{\bar\rho'}{r_e+\bar\rho}\cos^2(\theta)\, d\theta \\
\sqrt{2}\int _0^{2\pi}\frac{\bar\rho'}{r_e+\bar\rho}\sin^2(\theta)\, d\theta &\sqrt{2}\int
_0^{2\pi}\frac{\bar\rho'}{r_e+\bar\rho}\cos(\theta)\sin(\theta)\, d\theta
\end{bmatrix}\\=\begin{bmatrix} 1&0\\0&1\end{bmatrix}+O\bigl(\| \bar\rho\|
_{h^{1+\alpha}_p}\bigr) 
\end{multline*}
and observing that it is invertible for small $\bar\rho$, the system reads
\begin{equation}\label{nf}\begin{cases}
 \dot v=M(\bar\rho)^{-1}\begin{bmatrix}\pi _1^cG(\bar\rho)\\\pi
   ^s_1G(\bar\rho)\end{bmatrix}&\\
 \dot{\bar\rho}=\pi _1^\perp G(\bar\rho)-\pi
 _1^\perp\Big\{\frac{\bar\rho'}{r_e+\bar\rho}\tau _e\cdot
 M(\bar\rho)^{-1}\begin{bmatrix}\pi _1^cG(\bar\rho)\\\pi 
   ^s_1G(\bar\rho)\end{bmatrix}\Big\}=\pi _1^\perp \widetilde G(\bar\rho)=\pi _1^\perp
 G(\bar\rho)+O\bigl(\|\bar\rho\| _{h^{1+\alpha}_p}^2 \bigr)\, .& 
\end{cases}
\end{equation}
Notice again that the vector field in \eqref{nf} only depends on $\bar\rho$ and that, by
construction, 
\begin{multline*}
 \pi_1^\perp D\widetilde G(0)h=\pi_1^\perp DG(0)h=-F'(1)\frac{4V_0}{\pi r_e^4}\Bigl(\hat
 h_0+\sum_{k\geq 2}k\bigl[\hat h^c_k\cos(k\theta)+\hat h^s_k\sin(k\theta)\bigr]\Bigr)\, ,\\ h\in
 \pi _1^\perp h^{2+\alpha}_p=:h^{2+\alpha}_{p,\perp}\, .
\end{multline*}
It is not difficult to see that $\pi_1^\perp DG(0)\in
\mathcal{H}(h^{2+\alpha}_{p,\perp},h^{1+\alpha}_{p,\perp})$, either by applying
\cite[Theorem 41]{EG13} or by applying Fourier multiplier results such as those found
in \cite{AB04} for ``periodic'' symbols combined with a spectral reduction argument 
to split off the kernel, or by a direct computation of the associated semigroup and
Fourier multiplier results.\\
It follows that the principle of linearized stability \cite[Theorem 9.1.2]{Lun95} 
applies and yields local asymptotic stability of the trivial solution $\bar\rho\equiv 0$ of
$\dot{\bar\rho}=\pi _1^\perp \widetilde G(\bar\rho)$. Since the right-hand-side of
\eqref{nf} only depends on $\bar\rho$, since it is a smooth function of its argument,
and since $G(0)=0$, it follows that 
$$
 v(t)=v(0)+\int _0^t M\bigl(\bar\rho(\tau)\bigr)^{-1}\begin{bmatrix}\pi
   _1^cG\bigl(\bar\rho(\tau)\bigr)\\\pi
   ^s_1G\bigl(\bar\rho(\tau)\bigr)\end{bmatrix}\, d\tau\longrightarrow v_\infty\text{
   as }t\to\infty\, .
$$
The convergence is exponential since $\bar\rho$ converges to zero exponentially if
$\bar\rho _0$ is small enough, which is always the case provided $\rho _0$ is.
\begin{thm}
Let $\bigl(u_e,\mathbb{B}(x_e,r_e)\bigr)\in \mathcal{E}$ be an equilibrium solution of
\eqref{deq}. Then, for $\Gamma _0$ close enough to $\mathbb{S}_{r_e}$, the solution
$\bigl(u(\cdot,\Gamma _0),\Omega(\cdot,\Gamma _0)\bigr)$ exists globally and there 
exists 
$$
\rho(\cdot,\Gamma _0)\in\operatorname{C}^1\bigl([0,\infty),h^{1+\alpha}_p\bigr)\cap
\operatorname{C}\bigl([0,\infty), h^{2+\alpha}_p\bigr)
$$
with 
$$
\Omega (t,\Gamma_0)=\Omega_{\rho(t,\Gamma_0)}\text{ for }t\in[0,\infty)\, ,
$$
as well as $v_\infty=v_\infty(\Gamma_0)\in \mathbb{R} ^2$ such that
$$
 \bigl( v_1(\rho),v_2(\rho),\bar\rho(\rho)\bigr)\longrightarrow v_\infty\text{ as
 }t\to\infty\, ,
$$
exponentially fast. In other words the manifold $\mathcal{E}$ of equilibria is
locally asymptotically stable and any solution, starting close to it, converges
exponentially fast to a specific $\bigl(u_\infty,\mathbb{B}(v_\infty,r_e)\bigr)$
which depends only on the initial condition.
\end{thm}
\begin{proof}
The details of the argument are given in the discussion preceding the theorem for a
single fixed equilibrium. The same local analysis is, however, valid about any other
steady-state due to the translation invariance of the problem.
\end{proof}
It only remains to verify that $(v_1,v_2,\bar\rho)$ indeed provide a well-defined
coordinate system for $h^{2+\alpha}(\mathbb{S}_{r_e})$ about $\rho\equiv 0$.
\begin{prop}
For any given $\rho\in \mathcal{V}$ small enough, there is a unique small $v\in
\mathbb{R} ^2$ such that
$$
 \Gamma _\rho=\begin{bmatrix} v_1\\v_2\end{bmatrix}+\Gamma _{\bar\rho}\, ,
$$
for $\bar\rho\in N\bigl(DG(0)\bigr)^\perp$.
\end{prop}
\begin{proof}
Given $\rho\in \mathcal{V}$ small enough, i.e. such that $\Gamma_\rho\sim
\mathbb{S}_{r_e}$, a function $\bar\rho\in N\bigl(DG(0)\bigr)^\perp$ needs to be
found such that
$$
 \bigl[ r_e+\rho(\theta)\bigr]\begin{bmatrix} \cos(\theta)\\
   \sin(\theta)\end{bmatrix}=\bigl[
 r_e+\bar\rho(\varphi)\bigr] \begin{bmatrix}\cos(\varphi)\\ 
 \sin(\varphi)\end{bmatrix}+\begin{bmatrix} v_1\\v_2\end{bmatrix}\, .
$$
This identity implies that
\begin{align}\label{eq1}
  \bar\rho(\varphi)&=-r_e+\bigl[ (r_e+\rho)^2+|v|^2-2(r_e+\rho)v\cdot\nu
  _e\bigr]^{1/2}\, ,\\\label{eq2}
  \tan(\varphi)&=\frac{(r_e+\rho)\sin(\theta)-v_2}{(r_e+\rho)\cos(\theta)-v_1}\, ,\\\label{eq3}
  \cot(\varphi)&=\frac{(r_e+\rho)\cos(\theta)-v_1}{(r_e+\rho)\sin(\theta)-v_2}\, ,
\end{align}
where $\varphi=\varphi(\theta)$, and the  freedom of choice between representations
\eqref{eq2} and \eqref{eq3} will be exploited below. It is also useful to have
\begin{align}\label{eq4}
  (1+\tan ^2\varphi)\varphi '(\theta)&=\frac{\rho'v\cdot\nu
    _e+(r_e+\rho)^2-(r_e+\rho)v\cdot\tau
    _e}{\bigl((r_e+\rho)\cos(\theta)-v_1\bigr)^2}\, ,\\\label{eq5}
 (1+\cot ^2\varphi)\varphi '(\theta)&=\frac{\rho'v\cdot\tau
    _e+(r_e+\rho)^2-(r_e+\rho)v\cdot\nu
    _e}{\bigl((r_e+\rho)\sin(\theta)-v_2\bigr)^2}\, .
\end{align}
At this point only $v$ needs to be determined. This is done by requiring that the
following orthogonality conditions be satisfied
\begin{multline*}
\Phi _g(v_1,v_2)=\int _0^{2\pi}\bar\rho(\varphi)g(\varphi)\, d\varphi=\int
_{\varphi(0)}^{\varphi(2\pi)}\bar\rho(\varphi)g(\varphi)\, d\varphi \\=\int
_0^{2\pi}\bar\rho\bigl(\varphi(\theta)\bigr)g\bigl(\varphi(\theta)\bigr)\varphi
'(\theta)\, d\theta=0\text{ for }g=\sin,\cos\, .
\end{multline*}
Next the Jacobian of $\Phi=\begin{bmatrix} \Phi _{\cos}\\\Phi
  _{\sin}\end{bmatrix}$ is computed and is shown to be non-singular for
$\delta<<1$, which then implies the claim and concludes the proof. In order to 
compute $\partial _{v_1}\Phi$, representations \eqref{eq3}/\eqref{eq5} turn out to be
more convenient and leads to
\begin{multline*}
 \partial _{v_1}\Phi _g(0,0)=-\int _0^{2\pi}\frac{1}{{2}}
 \frac{{2}{(r_e+\rho)}\cos(\theta)}{{r_e+\rho}}
 \frac{{(r_e+\rho)^2}}{{\bigl(1+\cot
   ^2(\theta)\bigr)}{(r_e+\rho)^2}{\sin ^2(\theta)}}g(\theta)\,
 d\theta+\\
 \int _0^{2\pi}\rho
 \frac{{(r_e+\rho)^2}}{{\bigl(1+\cot ^2(\theta)\bigr)^2}}
 \frac{2\cos(\theta)}{{(r_e+\rho)^2\sin^3(\theta)}}
 \frac{1}{(r_e+\rho){\sin(\theta)}}g(\theta)\, d\theta+\\-\int _0^{2\pi}\rho 
 \frac{1}{1+\cot ^2(\theta)}\frac{\rho
   '\sin(\theta)+(r_e+\rho)\cos(\theta)}{(r_e+\rho)^2\sin ^2(\theta)}g(\theta)\,
 d\theta\\+\int _0^{2\pi}\rho \frac{(r_e+\rho)^2}{1+\cot
   ^2(\theta)}\frac{1}{(r_e+\rho)^2\sin ^2(\theta)}\Bigl(\partial _{v_1}\big |
 _{(0,0)}g(\varphi)\Bigr)\, d\theta\\=\int _0^{2\pi}\Big\{-\cos(\theta)+2
 \frac{\rho}{r_e+\rho}\cos(\theta)-\frac{\rho}{(r_e+\rho)^2}\bigl[\rho'\sin(\theta)
 +(r_e+\rho)\cos(\theta)\bigr]\Big\}g(\theta)\, d\theta +\int
 _0^{2\pi}\rho(\theta)g'(\theta)\, d\theta\, ,
\end{multline*}
where the last term is obtained by observing that
$$
 \partial _{v_1}\big | _{(0,0)}g(\varphi)=g'(\varphi)\varphi'\big | _{(0,0)}
 =g'(\theta)\frac{(r_e+\rho)^2}{\bigl(1+\cot ^2(\theta)\bigr)(r_e+\rho)^2\sin
   ^2(\theta)}=g'(\theta)\, ,\: \theta\in[0,2\pi)\, .
$$
It follows that
\begin{alignat*}{2}
 \partial _{v_1}\Phi _{\cos}(0,0)&=-1/2+O\bigl(\|\rho\| _{h^{1+\alpha}_p}\bigr)&&\text{
   as }\rho\to 0\, ,\\
 \partial _{v_1}\Phi _{\cos}(0,0)&=O\bigl(\|\rho\| _{h^{1+\alpha}_p}\bigr)&&\text{ as }\rho\to 0\, .
\end{alignat*}
In an analogous manner, but using representations \eqref{eq2}/\eqref{eq4} instead, it can be seen
that
\begin{alignat*}{2}
 \partial _{v_2}\Phi _{\cos}(0,0)&=O\bigl(\|\rho\| _{h^{1+\alpha}_p}\bigr)&&\text{
   as }\rho\to 0\, ,\\
 \partial _{v_2}\Phi _{\sin}(0,0)&=-1/2+O\bigl(\|\rho\| _{h^{1+\alpha}_p}\bigr)&&\text{ as }\rho\to 0\, ,
\end{alignat*}
so that the proof is complete.
\end{proof}
\begin{rem}
The reduction performed in this paper is similar to that required in the proof of the
generalized principle of stability \cite{PSZ09} already mentioned. Two essential
differences need, however, to be pointed out. On the one hand, the generalized
principle of linearized stability applies in general to abstract problems exhibiting
the right structure as identified in \cite{PSZ09}, while the use of nonlinear
coordinates in the space of shapes chosen here is responsible for a simpler normal
form and, thus, a much simpler proof. Since the geometric nonlinear coordinates are
tangent, at the steady-state, to the linear ones in which the linearization is
computed, the relevant linearized operator is actually the same.  
\end{rem}
\bibliography{../../lite}
\end{document}

%% file: commands.tex
\newcommand{\dhr}{\mathrel{\lhook\joinrel\relbar\kern-.8ex\joinrel\lhook\joinrel\rightarrow}}
